\documentclass{article}
\usepackage[]{graphicx}
\usepackage[]{amsmath,amssymb}
\usepackage[]{amsthm,enumerate}
\usepackage{verbatim,bm}
\usepackage{color}

\newtheorem{theorem}{Theorem}[section]
\newtheorem{lemma}[theorem]{Lemma}

\theoremstyle{definition}

\newtheorem{example}[theorem]{Example}

\theoremstyle{remark}
\newtheorem{remark}[theorem]{Remark}

\begin{document}

\title{A two-fold cover of strongly regular graphs with spreads and association schemes of class five}
\date{\today}

\author{
Sho Suda \\ {\small Department of Mathematics Education,  Aichi University of Education}\\ {\small 1 Hirosawa, Igaya-cho, Kariya, Aichi 448-8542, Japan} \\ {\small \texttt{suda@auecc.aichi-edu.ac.jp}}}

\maketitle
\abstract{A spread of strongly regular graphs is a partition of the vertex set by Delsarte cliques. 
We consider imprimitive association schemes of class four which are two-fold covers of strongly regular graphs with spreads. 
It will be shown that a two-fold cover of a strongly regular graph with a spread provides a five class fission scheme of the imprimitive scheme of class four.}

\section{Introduction}
A spread of a strongly regular graph is a set of disjoint cliques attaining the Delsarte bound.
This concept arises from a partial geometry. 
A spread in a partial geometry is a set of disjoint lines which cover all the points. 
The collinearity graph is a strongly regular graph and the spread corresponds to a set of disjoint cliques meeting the Delsarte bound. 

A strongly regular graph is a symmetric association scheme of class two.
If a strongly regular graph has a spread, then the set of edges is decompose into two sets depending on the edges being in a Delsarte clique in the spread or not. 
In \cite{B}, Brouwer showed that the strongly regular graphs obtained from partial geometries with spreads yield distance regular graphs with diameter $3$. 
In \cite[Proposition 4.1]{HT}, Haemers and Tonchev showed the same statement for strongly regular graphs with a spread  and conversely the three class association scheme with the same eigenmatrix has a fusion strongly regular graph with a spread. 

In this paper we consider a two-fold cover of strongly regular graphs with spreads. 
Here by a two-fold cover of an association scheme we mean  the imprimitive association schemes with a non-identity binary relation $R$ of valency one whose quotient by the equivalence relation the union of the identity relation and $R$ coincides with the given association scheme. 
We deal with the $4$-class case and consider the subsets whose quotient corresponds to Delsarte cliques in the strongly regular graph obtained by its quotient. 
The main theorem is that for $4$-class case, if the vertex set is decomposed into subsets whose quotient is a spread, then the scheme has a fission $5$-class association scheme. 

Finally we give examples obtained by the extended codes of dual of narrow-sense BCH codes of designed distance $5$. 
The code is a linear code with the all ones vector and with only $4$ weights and it is also an orthogonal array with strength at least $5$. 
The code has a structure of a $Q$-polynomial association scheme by a recent theorem of Bannai-Bannai in \cite{BB} applying to the binary Hamming schemes and its linear code having the all ones vector and strength $t$ as an orthogonal array and the degree $s$ satisfying $t\geq 2s-3$. 
The code can be decomposed into subsets whose quotient is a spread of the strongly regular graph. 
The main theorem implies that it yields a $5$-class association scheme.  
After the embedding of the vertex set by a suitable primitive idempotent, 
the partition corresponding to a spread is a union of disjoint cross polytopes and their antipodal ones. 
Such a spherical code is essentially same as mutually unbiased weighing matrices \cite{HKO,BKR}.  

\section{Preliminaries}
\subsection{Association schemes}
We first review the theory of association schemes.  
The reader is referred to \cite{BI} for more information. 
Let $X$ be a finite set of size $n$ and $R_i$ a nonempty binary relation on $X$ for $i=0,1,\ldots,d$.  
The pair $\mathfrak{X}=(X,\{R_i\}_{i=0}^d)$ is an association scheme of class $d$ if the following hold:
\begin{enumerate}
\item $R_0=\{(x,x)\mid x\in X\}$.
\item $\{R_i\mid i=0,1,\ldots,d\}$ is a partition of $X\times X$.
\item $\{(y,x)\mid (x,y)\in R_i\}=R_i$ for any $i=1,\ldots,d$.
\item For any $i,j,k\in \{0,1,\ldots,d\}$ there exists an integer $p_{i,j}^k$ such that for any $x,y\in X$ with $(x,y)\in R_k$, 
\begin{align*}
p_{i,j}^k=|\{z\in X\mid (x,z)\in R_i,(z,y)\in R_j\}|.
\end{align*} 
\end{enumerate}
The association schemes is usually called symmetric. 

For each binary relation $R_i$, we define $A_i$ as a $(0,1)$-matrix indexed by $X$ with $(x,y)$-entry $1$ if and only if $(x,y)\in R_i$. 
The condition for association schemes is characterized by the adjacency matrices. 
The pair $\mathfrak{X}=(X,\{R_i\}_{i=0}^d)$ is an association scheme if and only if the adjacency matrices satisfy the following 
\begin{enumerate}
\item $A_0=I_n$, where $I_n$ is the identity matrix of size $n$.
\item $\sum_{i=0}^d A_i=J_n$, where $J_n$ is the all ones matrix of size $n$.
\item $A_i^T=A_i$ for any $i=1,\ldots,d$, where $A_i^T$ denotes the transpose of $A_i$.
\item For any $i,j,k\in \{0,1,\ldots,d\}$ there exists an integer $p_{i,j}^k$ such that 
\begin{align*}
A_i A_j =\sum_{k=0}^d p_{i,j}^k A_k.
\end{align*}
\end{enumerate}
By the condition (iv), the adjacency matrices generate the algebra. 
This is called the Bose-Mesner algebra or the adjacency algebra, denoted by $\mathcal{A}$. 
Since $\mathcal{A}$ is generated by symmetric matrices,  $\mathcal{A}$ is commutative.
Thus there exists a basis consisting of primitive idempotens $E_0=\frac{1}{|X|}J,E_1,\ldots,E_d$. 
We then have change-of-bases matrices $P,Q$, called first and second eigenmatrices respectively, as follows:
\begin{align*}
A_i=\sum_{j=0}^d P_{j i}E_j, \quad E_i=\frac{1}{|X|}\sum_{j=0}^d Q_{j i}A_j.
\end{align*}
The values $k_i:=P_{0,i}$ and $m_i:=Q_{0,i}$ are called the valency of $A_i$ or $R_i$ and the multiplicity of $E_i$ respectively.

In this paper we focus on the association schemes satisfying following theorem.
\begin{theorem}\cite[Theorem 9.3]{BI}\label{thm:im}
The following are equivalent.
\begin{enumerate}
\item A graph $(X,R_i)$ is not connected for some $i\in\{1,\ldots,d\}$. 
\item $\sum_{i\in \mathcal{I}}A_i=p \sum_{j\in \mathcal{J}}E_j=I_q\otimes J_p$ for some integer $p,q\geq 2$ and $\mathcal{I}\subset \{0,1,\ldots,d\}$ and $\mathcal{J}\subset \{0,1,\ldots,d\}$
\end{enumerate} 
\end{theorem}
An association scheme is said to be imprimitve if one of the conditions in Theorem~\ref{thm:im} holds. 

For an imprimitive association scheme with a set $\mathcal{I}$ appearing in Theorem~\ref{thm:im}, we define a quotient scheme with respect to $\mathcal{I}$. 
The set $\mathcal{I}$ is an equivalence relation on $\{0,1,\ldots,d\}$ by $j\sim k$ if and only if $p_{i,j}^k\neq 0$ for some $i\in \mathcal{I}$.  
Let $\mathcal{I}_0=\mathcal{I},\mathcal{I}_1,\ldots,\mathcal{I}_t$ be the equivalent classes on $\{0,1,\ldots,d\}$ by $\sim$. 
Then by \cite[Theorem~9.4]{BI} there exist $(0,1)$-matrices $\overline{A}_j$ ($0\leq j\leq t$) such that 
\begin{align*}
\sum_{i\in \mathcal{I}_j}A_i=\overline{A}_j\otimes J_p,
\end{align*}
and the matrices $\overline{A}_j$ ($0\leq j\leq t$) define an association scheme $\overline{\mathfrak{X}}$. 
The association scheme $\overline{\mathfrak{X}}$ is called the quotient association scheme of $\mathfrak{X}$ with respect to $\mathcal{I}$ or equivalently $\cup_{i\in \mathcal{I}}R_i$. 
The binary relation whose adjacency matrix is $\overline{A}_i$ is denoted by $\overline{R}_i$. 
   
If $\mathcal{I}=\{0,i\}$ for some $i$ and the valency  of $R_i$ is one, then we say the association scheme $\mathfrak{X}$ is a two-fold cover of $\overline{\mathfrak{X}}$.  
   
\subsection{A two-fold cover of strongly regular graphs}
In this section, we consider imprimitive association schemes of class $4$, which is a two-fold cover of a strongly regular graph. 
We will see the relation between their second eigenmatrices. 

Let $\mathfrak{X}=(X,\{R_i\}_{i=0}^4)$ be an imprimitive association scheme such that there is a binary relation $R_j$ ($1\leq j\leq 4$) satisfying  the valency of $R_j$ is one and $R_0\cup R_j$ is an equivalence relation on the set of indices.
Let us rearrange the indices so that $R_4$ is the binary relation above and the equivalent classes on the indices are $\mathcal{I}_0=\{0,4\}$, $\mathcal{I}_1=\{1,3\}$, $\mathcal{I}_2=\{2\}$. 

Since the association scheme $(X,\{R_i\}_{i=0}^4)$ is imprimitive, there exists a quotient association scheme with respect to the equivalence relation $R_0\cup R_4$.  
Let  $\overline{\mathfrak{X}}=(\overline{X},\{\overline{R}_i\}_{i=0}^2)$ be the quotient association scheme.
We denote adjacency matrices and primitive idempotents by $\overline{A}_i$ and $\overline{E}_i$ ($0\leq i\leq 2$) respectively.
Renumber the indices on adjacency matrices of the quotient scheme as follows: 
\begin{align*}
A_0+A_4=\overline{A}_0\otimes J_2,\quad A_1+A_3=\overline{A}_1\otimes J_2,\quad A_2=\overline{A}_2\otimes J_2.
\end{align*}
Then $A_0+A_4,A_1+A_3,A_2$ generate a subalgebra of the Bose-Mesner algebra of $\mathfrak{X}$. 
Letting $\overline{E}_0,\overline{E}_1,\overline{E}_2$ be the primitive idempotents of $\overline{\mathfrak{X}}$, we rearrange the ordering the primitive idempotents of $\mathfrak{X}$ as follows: 
\begin{align*}
E_0=\frac{1}{2}\overline{E}_0\otimes J_2,\quad E_1=\frac{1}{2}\overline{E}_1\otimes J_2,\quad E_2=\frac{1}{2}\overline{E}_2\otimes J_2.
\end{align*}

From the binary relation $R_4$, we have the following natural bijection on $X$.
\begin{lemma}\label{lem:bij}
There exists a bijection $\phi:X\rightarrow X$ such that
if $(x,y)\in R_i$ $(x,\phi(y))\in R_{4-i}$ for any $i\in\{0,1,\ldots,4\}$ and any $x\in X$.
\end{lemma}
We extend this map $\phi$ to $\mathbb{R}^X=\mathbb{R}^{2n}$ or the matrix algebra of size $2n$ over $\mathbb{R}$, denoted by also $\phi$, as the multiplication of $I_{n}\otimes(J_2-I_2)$ from the left.

Let $Q$ ($\overline{Q}$ resp.) be the second eigenmatrix of $\mathfrak{X}$ ($\overline{\mathfrak{X}}$ resp.).
Set $\overline{Q}$ as 
\begin{align*}
\overline{Q}=\begin{pmatrix}
1&m &n-m-1\\
1&r &-r-1\\
1&s &-s-1
\end{pmatrix}
\end{align*} 
with $r>s$.
We have then the following theorem which provides a relation between second eigenmatrices of $\mathfrak{X}$ and $\overline{\mathfrak{X}}$.
\begin{theorem}
The second eigenmatrix $Q$ of $\mathfrak{X}$ is given as follows:
\begin{align*}
Q=\begin{pmatrix}
1&m &n-m-1 &m_3&m_4\\
1&r &-r-1& m_3 \alpha_3&m_4 \alpha_4\\
1&s &-s-1&0&0\\
1&r &-r-1& -m_3 \alpha_3&-m_4 \alpha_4\\
1&m &n-m-1 &-m_3&-m_4
\end{pmatrix}, 
\end{align*}
where $m_3,m_4,\alpha_3,\alpha_4$ satisfy the equations $m_3+m_4=n$, $m_3\alpha_3+m_4\alpha_4=0$, $k\alpha_3 \alpha_4=-1$ and $k$ is the valency of $R_1$.  
\end{theorem}
\begin{proof}
From the block form of $E_i$ for $i=0,1,2$ we have 
\begin{align*}
E_i&=\overline{E}_i\otimes \frac{1}{2}J_2\\
&=\frac{1}{2n}(\sum_{j=0}^2\overline{Q}_{j i}A_j)\otimes J_2\\
&=\frac{1}{2n}\sum_{j=0}^2\overline{Q}_{j i}(A_j\otimes J_2)\\
&=\frac{1}{2n}(Q_{0i}(A_0+A_4)+Q_{1i}(A_1+A_3)+Q_{2i}A_2), 
\end{align*}
 which implies first three columns of $Q$ with the ordering of adjacency matrices $A_0,A_1,\ldots,A_4$ are determined as desired. 
 
By the formula of $\overline{E}_i$ ($i=0,1,2$), we have 
\begin{align*}
E_0 \mathbb{R}^{2n}+E_1 \mathbb{R}^{2n}+E_2 \mathbb{R}^{2n}=\{(x,x)^T\mid x\in\mathbb{R}^n\}. 
\end{align*} 
Thus the eigenspaces $E_3, E_4$ satisfy 
 \begin{align*}
E_3 \mathbb{R}^{2n}+E_4 \mathbb{R}^{2n}=\{(x,-x)^T\mid x\in\mathbb{R}^n\}. 
\end{align*} 
Let $(x,-x)^T$ be an eigenvector of $A_i$ with eigenvalue $a$. 
Applying $\phi$ to the equation $A_i(x,-x)^T=a(x,-x)^T$, we have $A_{4-i}(x,-x)^T=-a(x,-x)^T$ because  $\phi(A_i)=A_{4-i}$ and $\phi(x,-x)^T=-(x,-x)^T$. 
By the equations $Q_{ij}/m_j=P_{ji}/k_i$ \cite[Theorem~3.5]{BI} we may set the third and fourth columns of $Q$ as in the theorem.  
We now check that these parameters $m_3,m_4,\alpha_3,\alpha_4$ satisfy the equations in theorem.

Compare $(1,1)$-, $(2,1)$-entries of the equation $Q P=2n I$, we have $m_3+m_4=n$ and $m_3\alpha_3+m_4\alpha_4=0$.

Let $\Delta_k$ and $\Delta_m$ be the diagonal matrices whose $i$-th entry is $k_i$ and $m_i$ respectively. 
Then $Q^T \Delta_k Q=|X|\Delta_m$ holds.
Comparing $(4,5)$-entry of its equation yields $k\alpha_3 \alpha_4=-1$. 
This completes the proof.
\end{proof}

\section{Complete bipartite subsets and their decomposition}
Let $Y$ be a nonempty subset of $X$ and let $\chi$ be the characteristic vector of $Y$. 
For a subset $I$ of $\{0,1,\ldots,4\}$ containing $0$, we say that $Y$ is a $I$-clique if 
\begin{align*}
\{i\mid R_i\cap(Y\times Y)\neq\emptyset\}=I.
\end{align*}
We will consider $\{0,2,4\}$-cliques of the association schemes $\mathfrak{X}=(X,\{R_i\}_{i=0}^4)$ and decompositions of the vertex set $X$ by such cliques. 

First we see the upper bound of $\{0,2,4\}$-cliques in the next lemma. 
The proof is due to the Hoffman-Delsarte clique bound.
\begin{lemma}\label{lem:clique}
Let $Y$ be a $\{0,2,4\}$-clique. 
Then $|Y|\leq 2(1-P_{0,2}/\theta)$ holds where $\theta=\min\{P_{1,2},P_{2,2}\}$.
If equality holds, then the size of $\{y \in Y\mid (y,x)\in R_i\}$ depends only on $i$, not on the particular choice of $x\in  X\setminus Y$. 
\end{lemma}
\begin{proof}
Let $\overline{Y}$ be a subset of $\overline{X}$ corresponding to $Y$ by the quotient map from $X$ to $\overline{X}$. 
Then $Y$ is a $\{0,2\}$-clique and it clearly holds that $|Y|\leq 2|\overline{Y}|$. 
Applying the Hoffman-Delsarte clique bound to $\{0,2\}$-clique $\overline{Y}$, we have $\overline{Y}\leq 1-P_{0,2}/\theta$. 
Thus we obtain the inequality $|Y|\leq 2(1-P_{0,2}/\theta)$. 

Equality holds if and only if $|Y|=2|\overline{Y}|$ and $|\overline{Y}|=1-P_{0,2}/\theta$ hold. 
The condition $|Y|=2|\overline{Y}|$ is equivalent to the valency of $R_2$ at each point in  $Y$ being $1$. 
The condition $|\overline{Y}|=1-P_{0,2}/\theta$ is equivalent to $\overline{Y}$ being a Delsarte clique with respect to the binary relation $\overline{R}_2$. 

From the quotient map from $X$ to $\overline{X}$, we have the following equalities for any $x\in X$ and its quotient image $\overline{x}\in\overline{X}$:
\begin{align}
& |R_1(x)\cap Y|+|R_3(x)\cap Y|=2|\overline{R}_1(\overline{x})\cap \overline{Y}|, \label{eq:R1}\\
& |R_2(x)\cap Y|=2|\overline{R}_2(\overline{x})\cap \overline{Y}| \label{eq:R2}
\end{align}

Since $R_2$ gives a matching on $Y$, $Y=\phi(Y)$ holds.  
Thus, for any $x\in X$, we have 
\begin{align}\label{eq:R3}
|R_1(x)\cap Y|=|\phi(R_1(x)\cap Y)|=|\phi(R_{1}(x))\cap \phi(Y)|=|R_{3}(x)\cap Y|.
\end{align}
Since $\overline{Y}$ is a Delsarte clique, it is a complete regular code in the strongly regular graph, which is same as the quotient scheme $\mathfrak{X}$, namely it holds that, for any $\overline{x}\in \overline{X}\setminus\overline{Y}$ and $i\in\{1,2\}$, 
$|\overline{R}_i(\overline{x}) \cap\overline{Y}|$ is independent of the choice of $\overline{x}$.
This properties of \eqref{eq:R1}, \eqref{eq:R2} and \eqref{eq:R3} imply that for any $x\in X\setminus Y$ and $i=1,2,3$, $|R_i(x)\cap Y|$ is independent of the choice of $x$.
\end{proof}

Next we consider the situation that the vertex set $X$ is decomposed into disjoint maximal $\{0,2,4\}$-cliques. 
Let $Y_1,\ldots,Y_f$ be $\{0,2,4\}$-cliques attaining the bound in Lemma~\ref{lem:clique} such that $X=\cup_{i=1}^f Y_i$ and $Y_i\cap Y_j=\emptyset$ for any distinct $i,j$. 
According to the partition $\{Y_1,\ldots,Y_f\}$ of $X$, we define the following refinement $\tilde{R}_i$ ($0\leq i\leq 5$) of the binary relations:
\begin{align*}
\tilde{R}_i&=R_i \text{ for } i=0,1,3,4,\\
\tilde{R}_2&=R_2\cap \bigcup_{i=1}^f (Y_i\times Y_i),\\
\tilde{R}_5&=R_2\setminus \tilde{R}_2.
\end{align*}
We denote by $\tilde{A}_i$ the adjacency matrices corresponding to $\tilde{R}_i$ for $i=0,1,\ldots,5$.  
The following is the main theorem, which provides that the refinement defines a fission association scheme of $\mathfrak{X}$. 
\begin{theorem}\label{thm:asc5}
The pair $\tilde{\mathfrak{X}}=(X,\{\tilde{R}_i\}_{i=0}^5)$ is an association scheme.
\end{theorem}
\begin{proof}
Let $\tilde{\mathcal{A}}$ be the vector space $\text{span}(\tilde{A}_0,\tilde{A}_1,\ldots,\tilde{A}_5)$ over $\mathbb{R}$. 
In order to show that $\tilde{\mathcal{A}}$ is closed under the matrix multiplication, it is enough to check that $\tilde{A}_2\tilde{A}_i$ and $\tilde{A}_5\tilde{A}_i$ are in $\tilde{\mathcal{A}}$ for $i=1,\ldots,5$.  

First it is obvious that $\tilde{A}_2\tilde{A}_4=\tilde{A}_2$ because $\tilde{A}_2$ corresponds to a partition  of $\{0,2,4\}$-cliques.

By Lemma~\ref{lem:clique} it holds that $(\tilde{A}_0+\tilde{A}_2+\tilde{A}_4)\tilde{A}_i\in \text{span}( \tilde{A}_0,\tilde{A}_2,\tilde{A}_4)$ for $i=1,3,5$.
Since the products of any of two in $\{\tilde{A}_1,\tilde{A}_3,\tilde{A}_4\}$ are in $\tilde{\mathcal{A}}$, we have  $\tilde{A}_2\tilde{A}_i\in \text{span}( \tilde{A}_0,\tilde{A}_2,\tilde{A}_4)$ for $i=1,3,5$. 
The matrix $\tilde{A}_2$ has the block form $(J_l\otimes I_f)\otimes J_2$, from which we routinely obtain that  $\tilde{A}_2^2$ is a linear combination of $\tilde{A}_0,\tilde{A}_2,\tilde{A}_4$. 
Thus we have shown 
\begin{align}\label{eq:1}
\tilde{A}_2\tilde{A}_i \in \tilde{\mathcal{A}} \text{ for } i=1,\ldots,5
\end{align} 

And it follows from that $J=\sum_{i=0}^5\tilde{A}_i$ and \eqref{eq:1} that $\tilde{A}_5 \tilde{A}_i$ is in $\tilde{\mathcal{A}}$ for $i=0,1,\ldots,5$.
This completes the proof. 
\end{proof}

Let us denote the second eigenmatrix of $\tilde{\mathfrak{X}}$ by $\tilde{Q}$. 
Then the second eigenmatrix $\tilde{Q}$ is determined by $Q$.
\begin{theorem}\label{thm:class5}
The second eigenmatrix $\tilde{Q}$ of $\tilde{\mathfrak{X}}$ has the following form:
\begin{align*}
\tilde{Q}=\begin{pmatrix}
1&m &n-m-1+\frac{m}{s}&m_3&m_4 & -\frac{m}{s}\\
1&r & -r& m_3 \alpha_3 & m_4 \alpha_4 & -1\\
1&s & -s & 0 & 0 & -1\\
1&r & -r & -m_3 \alpha_3 & -m_4 \alpha_4 &-1\\
1&m &n-m-1+\frac{m}{s} &-m_3&-m_4 &-\frac{m}{s}\\
1&s & -s-1+\frac{m}{s} &0&0 &-\frac{m}{s}
\end{pmatrix}. 
\end{align*}
\end{theorem}
\begin{proof}
We set $\tilde{E}_i$ by 
\begin{align*}
\tilde{E}_i&=E_i \quad (i=0,1,3,4),\\
\tilde{E}_5&=-\frac{m}{s}(J_{-\frac{m}{s}}-I_{-\frac{m}{s}})\otimes J_2-J_{2n},\\
\tilde{E}_2&=E_2-\tilde{E}_5.  
\end{align*}
It is clear that the sum of all $\tilde{E}_i$ equals to the identity matrix.
We will show that 
\begin{enumerate}
\item $E_2\tilde{E}_5=\tilde{E}_5$.
\item $\tilde{E}_i\tilde{E}_5=\delta_{i,5}\tilde{E}_5$ for $0\leq i\leq 5$.
\end{enumerate}
If the above conditions hold, then $\tilde{E}_i\tilde{E}_2=\delta_{i,2}\tilde{E}_2$ for $0\leq i\leq 5$ and thus  the set $\{\tilde{E}_0,\tilde{E}_1,\ldots,\tilde{E}_5\}$ forms the basis of primitive idempotents of the scheme $\tilde{\mathfrak{X}}$ and we have the desired second eigenmatrix with ordering of $\tilde{E}_0,\tilde{E}_1,\ldots,\tilde{E}_5$. 

We now show that (i) and (ii) hold.  
The matrix $\tilde{E}_5$ belongs to the subalgebra generated by $A\otimes J_2$ where $A$ is an element of the Bose-Mesner algebra of $\overline{\mathfrak{X}}$. 
By \cite[Proposition 4.1]{HT} the corresponding element of $\tilde{E}_5$ in $\overline{\mathfrak{X}}$, $E:=-\frac{m}{s}(J_{-\frac{m}{s}}-I_{-\frac{m}{s}})-J_{n}$, is a primitive idempotent of $\overline{\mathfrak{X}}$.
Thus (i) follows from $E\overline{E}_2=0$. 

Similarly we have $\tilde{E}_i\tilde{E}_5=0$ for $i=0,1$. 
And it is obvious that $\tilde{E}_5^2=\tilde{E}_5$.  
The image of $\tilde{E}_5$ is in $\{(x,x)\mid x\in\mathbb{R}^n\}$, while the images of $\tilde{E}_3$, $\tilde{E}_4$ are in $\{(x,-x)\mid x\in\mathbb{R}^n\}$.
Thus the images of $\tilde{E}_5$ and $\tilde{E}_i$ ($i=3,4$) are orthogonal, which implies that $\tilde{E}_i\tilde{E}_5=0$ for $i=3,4$. 
This proves (ii).
\end{proof}

\begin{remark}
\begin{enumerate}
\item The valencies of $\tilde{\mathfrak{X}}$ are 
\begin{align*}
(\tilde{k}_i)_{i=0}^5=(1,k,2(n-k-1)+\frac{m}{s},k,1,-\frac{m}{s}).
\end{align*} 
Using the formula $\tilde{Q}_{ij}/\tilde{m}_i=\tilde{P}_{ji}/\tilde{k}_j$ \cite[Theorem~3.5]{BI},  we obtain the first eigenmatrix $\tilde{P}$ of the scheme of $\tilde{\mathfrak{X}}$:
\begin{align*}
\tilde{P}=\begin{pmatrix}
1&k & 2(n-k-1)+\frac{m}{s}&k &1 & -\frac{m}{s}\\
1&\frac{k r}{m} & \frac{m+2s(n-k-1)}{m}& \frac{k r}{m} & 1  & -1\\
1& \tilde{P}_{2,1} & \frac{s(m+2s(n-k-1))}{m(s-1)+s(n-1)} & \tilde{P}_{2,1} & 1 & \frac{m(m-s(s+1))}{s(m(s-1)+s(n-1))}\\
1&\alpha k & 0 & -\alpha k  & -1 &0\\
1&-\frac{1}{\alpha} &0 &\frac{1}{\alpha}&-1 &0\\
1&\frac{k s}{m} & \frac{m+2s(n-k-1)}{m} &\frac{k s}{m}&1 &-\frac{m}{s}
\end{pmatrix}, 
\end{align*}
 where $\tilde{P}_{21}=\frac{k r s}{m(s-1)+s(n-1)}$
\item The association scheme $\tilde{\mathfrak{X}}$ is imprimitive with respect to the equivalence relation $\tilde{R}_0\cup\tilde{R}_4$. 
Its quotient scheme coincides with the association scheme appeared in \cite[Proposition 4.1]{HT} whose  spread is the quotient of the maximal $\{0,2,4\}$-clique decomposition supposed to exist. 
\end{enumerate}
\end{remark}

\section{Examples and mutually unbiased weighing matrices}
We give an obvious upper bound on the size of $\{0,2,4\}$-clique in terms of the multiplicities.
Moreover when equality holds, we connect the scheme of class five with mutually unbiased weighing matrices.
\begin{theorem}\label{thm:MUWM}
For the association scheme $\tilde{\mathfrak{X}}$ it holds that $\tilde{k}_0+\tilde{k}_4+\tilde{k}_5\leq \min\{2m_3,2m_4\}$.
If equality holds for $i=3$ or $4$, then there exists a set of mutually unbiased weighing matrices of size $-m/(2s)$ and  weight $1/\alpha_i$.
\end{theorem}
\begin{proof}
The second eigenmatrix shows that the embedding of the vertex set into eigenspaces corresponding to $\tilde{E}_3$ or $\tilde{E}_4$ is an antipodal spherical code with degree $4$. 
We denote this embedding by $\phi$.
In this embedding each $\{0,2,4\}$-clique is a set of vectors such that their usual inner products are $\pm1$, $0$. 
This means that it is a subset of a union of an orthonormal basis and its antipodal one, thus the size of a $\{0,2,4\}$-clique is at most $2m_i$ ($i=3,4$).
Therefore we have the desired bound. 

Assume that equality holds for $i=3$ or $4$.
The spherical code $\phi(X)$ obtained by $\tilde{E}_i$ is an antipodal set with the angle set $\{0,-1,\pm\alpha_i\}$. 
Each $\{0,2,4\}$-clique represents a union of an orthogonal basis and its antipodal one and a partition of the vertex set $X$ by $\{0,2,4\}$-cliques corresponds to a cross polytope decomposition. 
By \cite[Proposition 3.5]{NS} we obtain mutually unbiased weighing matrices of size $-m/(2s)$ and weight $1/\alpha_i$.  
\end{proof}
Finally we give an infinitely many examples fitting into Theorem~\ref{thm:MUWM}. 
\begin{example}
Let $m$ be a positive odd integer. 
Let $\mathcal{B}(2,m)$ be a narrow-sense BCH code with designed distance $5$.
Let $C$ be a code generated by the extended code of $\mathcal{B}(2,m)^\perp$ and the all-ones vector.
Then the dual code $\mathcal{B}(2,m)^\perp$ has the weights $\{0,2^{m-1}-2^{(m-1)/2},2^{m-1},2^{m-1}+2^{(m-1)/2},2^m\}$ \cite[Table 11.2]{HV}.
Since $\mathcal{B}(2,m)^\perp$ has a minimum distance at least $5$, 
$C$ is an orthogonal array of strength at least $5$. 
Moreover $C$ is a linear code containing all ones vector.

The binary Hamming scheme of length $2^m$ and the code $C$ have the property that there is a bijection mapping a codeword of weight $j$ to a codeword of weight $2^m-j$ and the code $C$ is closed under the bijection. 
By replacing the sphere to the binary Hamming scheme and the antipodal spherical design to the orthogonal array which is closed under the bijection above with strength $t$ and degree $s$ satisfying $t\geq 2s-3$,  we can apply \cite[Theorems~1.1, 1.2]{BB} to this case.
Then we obtain the scheme of class $4$ from the code $C$ with the second eigenmatrix $Q$ as follows:
\begin{align*}Q=
\begin{pmatrix}
1 &2^{m-1}(2^m-1)&(2^{m-1}+1)(2^m-1)&2^m&2^m(2^m-1)\\
1 &2^{m-1}&-2^{m-1}-1&2^{(m+1)/2}&-2^{(m+1)/2}\\
1 &-2^{m-1}&2^{m-1}-1&0&0\\
1 &2^{m-1}&-2^{m-1}-1&-2^{(m+1)/2}&2^{(m+1)/2}\\
1 &2^{m-1}(2^m-1)&(2^{m-1}+1)(2^m-1)&-2^m&-2^m(2^m-1)
\end{pmatrix}. 
\end{align*}
Here the scheme is $Q$-polynomial with respect to the ordering of primitive idempotents $E_0,E_3,E_1,E_4.E_5$.  The ordering of eigenspaces are arranged as same as the ordering in Theorem~\ref{thm:asc5}.

The code $C$ has the first order Reed-Muller code $RM(1,m)$ as a subcode whose size is $2^{m+1}$. 
Since $RM(1,m)$ has only the weight $\{0,2^{m-1},2^m\}$, it is a $\{0,2,4\}$-clique with the ordering of adjacency matrices determined in the second eigenmatrix above.
Moreover it attains the bound in Lemma~\ref{lem:clique}, where the values $2(1-P_{0,2}\theta)$ is $2^{m+1}$. 
The same is true for each coset of $C$ by $RM(1,m)$. 
Thus $C$ with cosets by $RM(1,m)$ has a fission scheme of class $5$ by Theorem~\ref{thm:asc5} and its second eigenmatrix is given as follows:
\begin{align*}\tilde{Q}=
\begin{pmatrix}
1&2^{m-1}(2^m-1)&2^{m-1}(2^m-1)&2^m&2^m(2^m-1)& 2^m-1\\
1&2^{m-1}&-2^{m-1}-1&2^{(m+1)/2}&-2^{(m+1)/2}&-1\\
1&-2^{m-1}&2^{m-1}-1&0&0&-1\\
1&2^{m-1}&-2^{m-1}-1&-2^{(m+1)/2}&2^{(m+1)/2}&-1\\
1&2^{m-1}(2^m-1)&2^{m-1}(2^m-1)&-2^m&-2^m(2^m-1)&2^m-1\\
1&-2^{m-1}&-2^{m-1}&0&0&2^m-1
\end{pmatrix}. 
\end{align*}
  
\end{example}

\end{document}